\documentclass[12pt, reqno]{amsart}
\usepackage{amssymb}
\usepackage{eucal}
\usepackage{amsmath}
\usepackage{amscd}
\usepackage[dvips]{color}
\usepackage{mathtools}
\usepackage{multicol}
\usepackage[all]{xy}           %xypic macro for latex2.09
\usepackage{graphicx}
\usepackage{color}
\usepackage{colordvi}
\usepackage{xspace}
\usepackage{ulem}
\usepackage{cancel}
\usepackage{tikz}
\usepackage{longtable}
\usepackage{multirow}
\usepackage{ifpdf}
\ifpdf
 \usepackage[colorlinks,final,backref=page,hyperindex]{hyperref}
\else
 \usepackage[colorlinks,final,backref=page,hyperindex,hypertex]{hyperref}
\fi

\begin{document}
\newcommand {\emptycomment}[1]{} %to remove paragraphs

\newcommand{\tabincell}[2]{\begin{tabular}{@{}#1@{}}#2\end{tabular}}

\newcommand{\nc}{\newcommand}
\newcommand{\delete}[1]{}

%%%%%%Use the next black to suppress labels
%\delete{
\nc{\mlabel}[1]{\label{#1}}  % Use this to suppress names
\nc{\mcite}[1]{\cite{#1}}  % Use this to suppress names
\nc{\mref}[1]{\ref{#1}}  % Use this to suppress names
\nc{\meqref}[1]{Eq.~\eqref{#1}} % Use this to suppress names
\nc{\mbibitem}[1]{\bibitem{#1}} % Use this to show number
%}

%%%%%%%%%%%Use the next block to show labels
\delete{
\nc{\mlabel}[1]{\label{#1}  % Use the next two lines to show names
{\hfill \hspace{1cm}{\bf{{\ }\hfill(#1)}}}}
\nc{\mcite}[1]{\cite{#1}{{\bf{{\ }(#1)}}}}  % Use this lines to show names
\nc{\mref}[1]{\ref{#1}{{\bf{{\ }(#1)}}}}  % Use this lines to show names
\nc{\meqref}[1]{Eq.~\eqref{#1}{{\bf{{\ }(#1)}}}} % Use this lines to show names
\nc{\mbibitem}[1]{\bibitem[\bf #1]{#1}} % Use this to show name
}

%%%%%%%%%%%%%%%%%%%%%%%% Statements
\newtheorem{thm}{Theorem}[section]
\newtheorem{lem}[thm]{Lemma}
\newtheorem{cor}[thm]{Corollary}
\newtheorem{pro}[thm]{Proposition}
\newtheorem{conj}[thm]{Conjecture}
\theoremstyle{definition}
\newtheorem{defi}[thm]{Definition}
\newtheorem{ex}[thm]{Example}
\newtheorem{rmk}[thm]{Remark}
\newtheorem{pdef}[thm]{Proposition-Definition}
\newtheorem{condition}[thm]{Condition}

\renewcommand{\labelenumi}{{\rm(\alph{enumi})}}
\renewcommand{\theenumi}{\alph{enumi}}
\renewcommand{\labelenumii}{{\rm(\roman{enumii})}}
\renewcommand{\theenumii}{\roman{enumii}}

\nc{\tred}[1]{\textcolor{red}{#1}}
\nc{\tblue}[1]{\textcolor{blue}{#1}}
\nc{\tgreen}[1]{\textcolor{green}{#1}}
\nc{\tpurple}[1]{\textcolor{purple}{#1}}
\nc{\btred}[1]{\textcolor{red}{\bf #1}}
\nc{\btblue}[1]{\textcolor{blue}{\bf #1}}
\nc{\btgreen}[1]{\textcolor{green}{\bf #1}}
\nc{\btpurple}[1]{\textcolor{purple}{\bf #1}}

%new commands

\newcommand{\End}{\text{End}}

\nc{\calb}{\mathcal{B}}
\nc{\call}{\mathcal{L}}
\nc{\calo}{\mathcal{O}}
\nc{\frakg}{\mathfrak{g}}
\nc{\frakh}{\mathfrak{h}}
\nc{\ad}{\mathrm{ad}}
\def \gl {\mathfrak{gl}}
\def \g {\mathfrak{g}}

\nc{\ccred}[1]{\tred{\textcircled{#1}}}

%\nc{\move}[1]{\footnote{#1}}

\newcommand{\cm}[1]{\textcolor{purple}{\underline{CM:}#1 }}

\newcommand\blfootnote[1]{%
  \begingroup
  \renewcommand\thefootnote{}\footnote{#1}%
  \addtocounter{footnote}{-1}%
  \endgroup
}

%%%%%%%%%%%%%%%%%%%%%%%%%%%%%%%%%%%%%%%%%%%%%%%%%%%%%%%%%%%%%%%%%%

%%%%%%%%%%%%%%%%%%%%%%%%%%%%%%%%%%%%%%%%%%%%%%%%%%%%%%%%%%%%%%%%%%%%%%%%%%%
%%%%%%%%%%%%%%%%%%%%%%    Title    %%%%%%%%%%%%%%%%%%%%%%%%%%%%%%%%%%%%%%%%
\title[Manin triples and bialgebras of Left-Alia algebras associated to invariant theory]
{Manin triples and bialgebras of Left-Alia algebras associated to invariant theory}
\author{Chuangchuang Kang}
\address{
School of Mathematical Sciences  \\
Zhejiang Normal University\\
Jinhua 321004 \\
China}
\email{kangcc@zjnu.edu.cn}

\author{Guilai Liu}
\address{Chern Institute of Mathematics \& LPMC    \\
Nankai University \\
 Tianjin 300071   \\
  China}
\email{1120190007@mail.nankai.edu.cn}

\author{Zhuo Wang}
\address{
School of Mathematical Sciences \& LPMC    \\
Nankai University \\
Tianjin 300317              \\
China}
\email{2013270@nankai.edu.cn}

\author{Shizhuo Yu}
\address{
School of Mathematical Sciences \& LPMC    \\
Nankai University \\
Tianjin 300317              \\
China}
\email{yusz@nankai.edu.cn}

%\dedicatory{\it{Dedicated to the Memory of Professor Yuri I. Manin (1937-2023)}}
\blfootnote{*Corresponding Author: Guilai Liu. Email: 1120190007@mail.nankai.edu.cn.}

\begin{abstract}
A left-Alia algebra is a vector space together with a bilinear map satisfying symmetric Jocobi identity.   Motivated by invariant theory, we first construct a class of left-Alia algebras induced by twisted derivations. Then, we introduce the notion of Manin triples and  bialgebras of left-Alia algebras. Via specific matched pairs of left-Alia algebras, we figure out the equivalence between Manin triples and  bialgebras of left-Alia algebras.
\end{abstract}

\subjclass[2010]{
    17A36,  %Automorphisms, derivations, other operators (nonassociative rings and algebras)
    17A40,  %Ternary compositions
    %18M70,  %Algebraic operads, cooperads, and Koszul duality
    %%%%17B:Lie algebras and Lie superalgebras
    17B10, %Representations, algebraic theory
    17B40, %Automorphisms,derivations,other operators
    17B60, %Lie (super)algebras associated with other structures (associative, Jordan, ect.)
    17B63,  %Poisson algebras
    17D25.  %Lie-admissible algebras
    %37J39,   %Relations of finite-dimensional Hamiltonian and Lagrangian systems with topology, geometry and differential geometry (symplectic geometry, Poisson geometry, etc.
    %53D17  %Poisson manifolds; Poisson groupoids and algebroids
}

\keywords{Left-Alia algebras, bialgebras, invariant theory}

\maketitle

%\vspace{-1.5cm}

\tableofcontents

\allowdisplaybreaks

\section{Introduction and Main Statements}
\subsection{Introduction}
Let $G$ be a finite group and  $\mathbb {K}$ an algebraic closed field  of characteristic zero.
Suppose that $V$ is an $n$-dimensional faithful representation of $G$ and
$S=\mathbb {K}[V]=\mathbb {K}[x_1,\ldots,x_n]$
is the coordinate ring of $V$.

The goal of invariant theory is to study the structures of the ring of invariants
\[S^{G}=\{f\in S: a\cdot f=f,\ \forall a\in G\},\]
in which the group action is extended from the representation of $G$ (see Section~\ref{Pre} for more details).
In particular, Hilbert proved that $S^G$ is always a finite generated $\mathbb {K}$-algebra~\cite{Hil} and Chevalley~\cite{Che}, Shephard and Todd~\cite{ST} proved that $S^G$ is a polynomial algebra if and only if~$G$ is generated by pseudo-reflections (see Section~\ref{twisted} for precise definition).

Twisted derivations ~\cite{sigma} (also named $\sigma$-derivations) play an important role in the study of deformations of Lie algebras. Motivated by the above Chevalley's Theorem, we apply pseudo-reflections to induce a class of twisted derivations on $S$ (see Theorem~\ref{main12} for more details). Based on twisted derivations on commutative associative algebras, we obtain a class of left-Alia (left anti-Lie-admissible) algebras~\cite{Dzh09}, which appears in the study of a special class of algebras with a skew-symmetric identity of degree three. Furthermore, we construct Manin triples and bialgebras of left-Alia  algebras. Via specific matched pairs of left-Alia algebras, we figure out the equivalence between Manin triples and bialgebras.

Throughout this paper, unless otherwise specified, all vector spaces are finite-dimensional over an algebraically closed field $\mathbb{K}$ of characteristic zero and all $\mathbb{K}$ -algebras are  commutative and associative with the finite Krull dimension, although many results and notions remain valid in the infinite-dimensional case.

\subsection{Left-Alia Algebras Associated with Invariant Theory}
The notion of  a left-Alia algebra  was defined for the first time in the table in the ~{Introduction} %MDPI: PLease check if this should be revised into Section 1.1. Also, please check if the table should be a citation. There is no table in the paper.
 of~\cite{Dzh09}. %Please check intended meaning has been retained
\begin{defi}[\cite{Dzh09}]
	A \textbf{left-Alia algebra} (also named a 0-Alia algebra) is a vector space $A$ together with a bilinear map $[\cdot,\cdot]:A\otimes A\rightarrow A$ satisfying the \textbf{symmetric Jacobi identity}: %MDPI: We added italic.
	%MDPI: Also, please check if bold is necessary. Same bellow. Same for the rest of the paper.
	\begin{equation}\label{eq:0-Alia}
		[[x, y], z]+[[y, z], x]+[[z, x], y]=[[y, x], z]+[[z, y], x]+[[x, z], y],\;\forall x,y,z\in A.
	\end{equation}
\end{defi}
There are some typical examples of left-Alia algebras. Firstly, when the bilinear map~$[\cdot,\cdot]$ is skew-symmetric, $(A,[\cdot,\cdot])$ is a Lie algebra. By contrast, any commutative algebra is a left-Alia algebra and, in particular,  a mock-Lie algebra~\cite{Zusmanovich} (also known as a Jacobi--Jordan algebra in~\cite{Burde}) with a symmetric bilinear map that satisfies the Jacobi identity is a   left-Alia algebra.
Secondly, the notion of an anti-pre-Lie algebra ~\cite{LB2022}  was recently studied as a left-Alia algebra with an additional condition. Anti-pre-Lie algebras are the underlying algebra structures of nondegenerate commutative $2$-cocycles~\cite{Dzh} on Lie algebras and are characterized as Lie-admissible algebras whose negative multiplication operators compose representations of commutator Lie algebras. Condition \eqref{eq:0-Alia} of the identities of an anti-pre-Lie algebra is just to guarantee $(A,[\cdot,\cdot])$ is a Lie-admissible algebra. Additionally, we also studied left-Alia algebras in terms of their relationships with Leibniz algebras~\cite{Loday} and Lie triple systems ~\cite{Helgason}.

Let $(A,\cdot)$ be a commutative associative algebra and $R:A\rightarrow A$ a linear map on $A$. For brevity, the operation $\cdot$ will be omitted. A linear map $D:A\rightarrow A$ is called \textbf{~{a twisted derivation}} %MDPI: Please check if bold is necessary. Otherwise, please remove it. Same for the following highghlights. Same for the rest of the paper.
with respect to an $R$ (also named a $\sigma$-derivation in~\cite{sigma}) if $D$ satisfies the twisted Leibniz rule:
\begin{equation}\label{eq:tw}
	D(fg)=D(f)g+R(f)D(g),\ f,g\in A.
\end{equation}

Non-trivial examples of twisted derivations can be constructed in invariant theory. In particular, each pseudo-reflection $R$ on a vector space $V$ induces a twisted derivation $D_R$~on the polynomial ring $\mathbb{K}[V]$ (see Section~\ref{twisted} for details).

Define
\[[f,g]_R=D(f)g-R(f)D(g).\]

We then obtain a class of left-Alia algebras in ~{Theorem A.}\\

\textbf{Theorem A} (Theorem \ref{main12} and Theorem \ref{thm:R-D}) \\
(a) {\it For each twisted derivation $D$ on $A$,  $(A,[\cdot,\cdot]_R)$ is a left-Alia algebra.}\\
(b) {\it Each pseodo-reflection $R$ on $V$ induces a left-Alia algebra $(\mathbb{K}[V],[\cdot,\cdot]_R)$.}\\

This applies when $R=I$, $[f,g]_R$ is skew-symmetric and $(A,[\cdot,\cdot]_R)$ is a Lie algebra of the Witt type~\cite{WT}. %Please check intended meaning has been retained
Moreover, ~{Theorem A} also provides a class of left-Alia algebras on polynomial rings from invariant theory. As a corollary of ~{Theorem A,} we see that when $S^G$ is a polynomial algebra, each generator $g\in G$ corresponds to a left-Alia algebra $(S,[\cdot,\cdot]_{R_g})$. The collection of left-Alia algebras is also an interesting research object for further study.

In addition, if we define that $[f,h]=D_R(f)h-fD_R(h)$ on $S$, $(S,[\cdot,\cdot])$, then it is not a left-Alia algebra in general. %Please check intended meaning has been retained
 However, when  $[\cdot,\cdot]$ is restricted to $V^*$, we obtain a finite-dimensional Lie algebra, which induces a linear Poisson structure on $V$, and figure out the entrance to the study of twisted relative Poisson structures %Please check intended meaning has been retained
 on graded algebras. See~\cite{KYZ} for reference.

\subsection{Manin Triples and Bialgebras of Left-Alia Algebras }
A  bialgebra structure is a vector space equipped with both an
algebra structure and a coalgebra structure satisfying certain
compatible conditions. Some well-known examples of such structures
include Lie bialgebras~\cite{Cha,Dri},  which are closely
related to Poisson--Lie groups and play an important role in the infinitesimalization of quantum
groups, and antisymmetric infinitesimal bialgebras~\cite{Agu2000, Agu2001,
	Agu2004,Bai2010,SW} as equivalent structures of double constructions of Frobenius algebras which are widely
applied in the 2d topological field and string theory~\cite{Kock,Lau}.
Recently, the notion of anti-pre-Lie bialgebras was \linebreak  studied in~\cite{LB2023}, which serves as a preliminary to supply a reasonable bialgebra theory for transposed Poisson algebras~\cite{Bai20232}.
The notions of  mock-Lie bialgebras~\cite{Benali} and Leibniz bialgebras~\cite{RSH,ShengT} were also introduced with different motivations.
These bialgebras have a common property in that they can be equivalently characterized by Manin triples which correspond to nondegenerate invariant bilinear forms on the algebra structures.
In this paper, we follow such a procedure to study left-Alia bialgebras.

To develop the bialgebra theory of left-Alia algebras, we first define a representation of a left-Alia algebra
to be a triple $(l,r,V)$, where $V$ is a vector space and $l,r:A\rightarrow\mathrm{End}(V)$ are linear maps such that the following equation holds:
\begin{equation*}
	l([x,y])v-l([y, x])v= r(x)r(y)v-r(y)r(x)v+r(y)l(x)v-r(x)l(y)v,\;\forall x,y\in A, v\in V.
\end{equation*}

A representation $(\rho,V)$ of a Lie algebra $(A,[\cdot,\cdot])$ renders representations $(\rho,-\rho,V)$ and  $(\rho,2\rho,V)$ of $(A,[\cdot,\cdot])$ as left-Alia algebras.

Furthermore, we introduce the notion of a quadratic left-Alia algebra, defined as a left-Alia algebra $(A,[\cdot,\cdot])$ equipped with a nondegenerate symmetric bilinear form~$\mathcal{B}$~which is invariant in the sense that
\begin{equation*}
	\mathcal{B}([x, y],z)=\mathcal{B}(x,[z, y]-[y, z]),\;\forall x,y,z\in A.
\end{equation*}

A quadratic left-Alia algebra gives rise to the equivalence between the adjoint representation and the coadjoint representation.

Last, we introduce the notions of a matched pair  (Definition \ref{match}) of left-Alia algebras, a Manin triple of left-Alia algebras ~{(Definition}~\ref{Manin})  %MDPI: Definition 11 is mentioned first, before all citations. Please make sure that Definitions, Theorems, etc. are in numeral order. Please cite them all if possible. Please check and revise. Same for the rest of the paper. Same for other highlights in the paper.
and a~left-Alia bialgebra (Definition \ref{bialgebra}).  Via specific matched pairs of left-Alia algebras, we figure out the equivalence between Manin triples and bialgebras in ~{Theorem B.}\\

\textbf{Theorem B} (Theorem \ref{thm:1} and Theorem \ref{thm:2})
 {\it Let $(A,[\cdot,\cdot] )$ be a left-Alia algebra. Suppose that there is a left-Alia algebra structure $(A^{*},\{\cdot,\cdot\} )$ on the dual space $A^{*}$, and $\delta:A\rightarrow A\otimes A$ is the dual of $\{\cdot,\cdot\}$. Then the following conditions are equivalent:
 \begin{enumerate}
     \item There is a Manin triple of left-Alia algebras $\big( (A\oplus A^{*},[\cdot,\cdot]_{d},\mathcal{B}_{d}), A, A^{*} \big)$,
        where,
 \begin{equation*}
[(x,u),(y,v)]_d=([x, y],l(x)v+r(y)u),\;\forall x,y\in A, u,v\in A^*,\ \ \ and
\end{equation*}
 \begin{equation*}
    \mathcal{B}_{d}((x,u),(y,v))=\langle x, v\rangle+\langle y,u\rangle,\;\forall x,y\in A, u,v\in A^{*}.
\end{equation*}

     \item $(A,[\cdot,\cdot] ,\delta)$ is a left-Alia bialgebra.
 \end{enumerate}
 }
\vspace{5mm}
Theorem B naturally leads to the study of Yang--Baxter equations and relative Rota--Baxter operators for left-Alia algebras~\cite{YBRB}.

\section{Pseudo-Reflections and Twisted Deviations in Invariant Theory}
\subsection{Preliminary on Invariant Theory}\label{Pre}
Let $G$ be a finite group and  $\mathbb {K}$ an algebraic closed field  of characteristic zero.
Suppose that $(\rho,V)$ is an $n$-dimensional faithful representation of $G$ and its dual representation is denoted by $(\rho,V^*)$.
Let $S=\mathbb {K}[V]=\mathbb {K}[x_1,\ldots,x_n]$ be the coordinate ring of $V$. Define a $G$-action on $S$ as
\begin{equation}\label{eq:action}
	g\cdot \sum_{i_1,\ldots,i_n} k_{i_1,\ldots,i_n} x_1^{i_1}\ldots x_n^{i_n}:=\sum_{i_1,\ldots,i_n} k_{i_1,\ldots,i_n}(\rho (g)x_1)^{i_1}\ldots (\rho (g)x_n)^{i_n},\;\forall g\in G.
\end{equation}

Define the ring of invariants as
\[S^{G}=\{f\in S: a\cdot f=f,\ \forall a\in G\}.\]
\begin{thm}[\cite{IT,Hil}]
	\begin{enumerate}
		\item[(a)]
		$S^{G}$ is a finitely generated $\mathbb {K}$-algebra.
		\item[(b)]
		$S$ is a finitely generated $S^G$-module.
	\end{enumerate}
\end{thm}

\begin{defi}[\cite{IT}]
	A linear automorphism $R\in Aut(V)$ is called a \textbf{{pseudo-reflection}} if $R^m=I$ for some $m\in \mathbb{N}^*$ and $Im(I-R)$ is one-dimensional.
\end{defi}
In invariant theory, the following theorem gives the equivalent condition that $S^G$ is a polynomial algebra:
\begin{thm}[\cite{Che,ST}]
	$S^G$ is a polynomial algebra if and only if $G\cong \rho(G)$ is generated by \linebreak  pseudo-reflections.
\end{thm}

Then, we figure out the relation between a pseudo-reflection on $V$ and a twisted deviation on $S=\mathbb{K}[V]$.
\begin{lem}\label{dual}
	Let $R$ be a pseudo-reflection on $V$. Then, $R$ induces a pseudo-reflection on $V^*$ (also denoted by $R$).
\end{lem}
\begin{proof}
	Let $\{e_1,\ldots,e_n\}$ be a basis of $V$ such that $W={\rm Span} \{e_1,\ldots,e_{n-1}\}$ is fixed by $R$.
	By $R^m=\left(
	\begin{array}{ccccc}
		1 &   &   &   & a_1 \\
		& 1 &   &   & a_2 \\
		&   & \ddots &   & \vdots   \\
		&   &   & 1 & a_{n-1} \\
		&   &   &   & a_n \\
	\end{array}
	\right)^m
	=I$, we see that $R$ is given by the diagonal matrix $diag(1,\ldots,1,\omega)$, where $\omega\neq 1$ is an $m$-th primitive root over $\mathbb{K}$. Denote $\{x_1,\ldots,x_n\}$, the dual basis of $V^*$, such that $x_i(e_j)=\delta_{ij}$. %Please check intended meaning has been retained
	Thus, the induced automorphism on $V^{*}$, defined by $R(x_i)(e_j):=x_i(R^{-1}(e_j))$, satisfies that
	$R(x_i)=x_i$, $1\leq i\leq n-1$ and $R(x_n)=(1/\omega) x_n$. Therefore, $R$ is a pseudo-reflection on $V^*$.
\end{proof}

\subsection{Pseudo-Reflections Induced by Twisted Deviations}\label {twisted}
Let $(A,\cdot)$ be a commutative associative algebra and $R:A\rightarrow A$ a linear map on $A$. Recall from~\cite{sigma} the definition of a twisted derivation (also named a $\sigma$-derivation).
\begin{defi}
	A linear map $D:A\rightarrow A$ is called a \textbf{{twisted derivation}} with respect to  $R$ if $D$ satisfies the twisted Leibniz rule:
	\begin{equation}\label{eq:s-der}
		D(fg)=D(f)g+R(f)D(g),\ f,g\in A.
	\end{equation}
\end{defi}
\begin{rmk}
	When $R=I$, $D$ is a derivation on $A$.
\end{rmk}

Recall from Section~\ref{Pre} that for a fixed non-zero $v_R\in Im (I-R)\subset V$, there exists a $\Delta_R\in V^*$ such that
\begin{equation}
	(I-R)v=\Delta_R(v)v_R,\ \ \forall v\in V.
\end{equation}

By Lemma~\ref{dual}, for a fixed non-zero $l_R\in Im (I-R)\subset V^*$, there also exists a $\Delta_R\in V$ such that
\begin{equation}\label{eq:ref}
	(I-R)x=\Delta_R(x)l_R,\ \ \forall x\in V^*.
\end{equation}

Also, denote $R:S\rightarrow S$ as an extension of  $R\in {\rm Aut} (V)$ satisfying
\[
R(k_1 f+k_2h)=k_1R(f)+k_2R(h)\ {\rm and}\  R(fh)=R(f)R(h).
\]

\begin{thm}\label{main12}
For each $f\in S$, there exists a twisted derivation $D_R:S\rightarrow S$ with respect to $R$
	such that
	\begin{equation}
		R(f)=f-D_R(f)l_R.
	\end{equation}
\end{thm}
\begin{proof}
	First, we prove that $R(f)$ can be uniquely written as $R(f)=f-D_R(f)l_R$ for some $D_R:S\rightarrow S$.
	It follows from (\ref{eq:ref}) that, for $1\leq a_i\leq n$,
	\begin{eqnarray*}
		R(x_{a_1}\ldots x_{a_k})&=&(Rx_{a_1})\ldots (Rx_{a_k})\\
		&=&(x_{a_1}-\Delta_R(x_{a_1})l_R)\ldots (x_{a_k}-\Delta_R(x_{a_k})l_R)
	\end{eqnarray*}
	which can be expressed as
	\[x_{a_1}\ldots x_{a_k}-D_R(x_{a_1}\ldots x_{a_k})l_R,\]
	where $D_R$ maps the monomial to a polynomial in $S$.
	As a consequence, $R(f)$ can be written as
	\begin{equation}
		R(f)=f-D_R(f)l_R,
	\end{equation}
	where $D_R:S\rightarrow S$ is a linear map.
	Then, we prove that $D_R$ is a twisted derivation on $S$ with respect to $R$.
	On the one hand,
	\[R(fh)=fh-D_R(fh)l_R.\]
	
	{On the} other hand, %MDPI: Indent added for new paragraphs. Please check and confirm. Same for the rest of the paper.
	\begin{eqnarray*}
		R(f)R(h)=R(f)(h-D_R(h)l_R)&=&(f-D_R(f)l_R)h-R(f)D_R(h)l_R\\
		&=&fh-(D_R(f)h+R(f)D_R(h))l_R.
	\end{eqnarray*}
	
	Therefore, $R(fh)=D_R(f)h+R(f)D_R(h)$.
\end{proof}
\begin{rmk}
	When restricting $D_R$ to $V^*$, $D_R=\Delta_R$ on $V^*$. When restricting $(I-D_R)$ to $V^*$, $(I-D_R)$ is a pseudo-reflection on $V^*$.
\end{rmk}

\section{Left-Alia Algebras and Their Representations}
\subsection{Left-Alia Algebras and Twisted Derivations}

\vspace{6pt}
\begin{defi}[\cite{Dzh09}]
	A \textbf{{left-Alia algebra}} is a vector space $A$ together with a bilinear \linebreak  map $[\cdot,\cdot]:A\times A\rightarrow A$ satisfying the symmetric Jacobi property:
	\begin{equation*}
		[[x, y], z]+[[y, z], x]+[[z, x], y]=[[y, x], z]+[[z, y], x]+[[x, z], y],\;\forall x,y,z\in A.
	\end{equation*}
\end{defi}
\begin{rmk}
	A left-Alia algebra $(A,[\cdot,\cdot])$ is a Lie algebra if and only if the bilinear map $[\cdot,\cdot]$ is skew-symmetric.
	On the other hand, any commutative algebra $(A,[\cdot,\cdot])$ in the sense that $[\cdot,\cdot]$ is symmetric is a left-Alia algebra.
\end{rmk}

We can obtain a class of left-Alia algebras from twisted derivations.

\begin{lem}
	Let $D:A\rightarrow A$  be a twisted derivation of the commutative associative algebra $(A,\cdot)$. \linebreak  Then, $D$ satisfies
	\begin{equation}\label{eq:d-use}
		xD(y)-D(x) y= R(x) D(y)-D(x) R(y),\quad\forall~x,y\in A.
	\end{equation}
\end{lem}
\begin{proof}
	By the commutative property of $(A,\cdot)$ and \eqref{eq:s-der}, we have
	\begin{eqnarray*}
		D(xy)-D(yx)=D(x)y+R(x)D(y)- D(y)x-R(y) D(x)=0.
	\end{eqnarray*}
	
Therefore, \eqref{eq:d-use} holds.
\end{proof}

\begin{thm}\label{thm:R-D}
	Let $(A,\cdot)$ be a commutative associative algebra and $D$ be a twisted derivation. For all $x,y\in A $, define the bilinear map $[\cdot,\cdot]_R:A\times A\rightarrow A$ by
	\begin{equation}\label{eq:R-der}
		[x,y]_R:=[x,y]_{D}=x D(y)-R(y) D(x).
	\end{equation}
	
Then, $(A,[\cdot,\cdot]_R)$ is a left-Alia algebra.
\end{thm}
\begin{proof} Let $x,y,z\in A$. By \eqref{eq:R-der}, we {have} %MDPI: Please check if the format of the equation citation over the ``='' could be revised. Same for following equations too.
	\begin{align*}
		[x,y]_R-[y,x]_R&~=xD(y)-R(y) D(x)-yD(x)+R(x) D(y)\\
		&\overset{\eqref{eq:d-use}}{=}2(x D(y)-yD(x)),
	\end{align*}
	and
	\begin{align*}
		D(xD(y)-yD(x))&\overset{\eqref{eq:s-der}}{=}D(x) D(y)-R(y) D^2(x)-D(y) D(x)+R(x) D^{2}(y)\\
		&=R(x)D^{2}(y)-R(y)\ D^2(x).
	\end{align*}
	
Furthermore,
	\begin{align*}
		&\circlearrowleft_{x,y,z}[[x,y]_R-[y,x]_R,z]_R\\
		=& \circlearrowleft_{x,y,z} 2[x D(y)-y D(x),z]_R\\
		\overset{\eqref{eq:R-der}}{=}&\circlearrowleft_{x,y,z} 2\big(x D(y)D(z)
		-y D(x) D(z)
		-R(x)D^2(y)R(z)+R(y) D^2(x) R(z)\big)\\
		=&~0.
	\end{align*}
	Therefore, the conclusion holds.
\end{proof}

\begin{rmk}
	Theorem~\ref{thm:R-D} can also be verified in the following way.
	Let $(A,\cdot)$ be a commutative associative algebra with linear maps $f,g:A\rightarrow A$.
	By~\cite{Dzh09}, there is a left-Alia algebra $(A,[\cdot,\cdot])$ given by
	\begin{equation}\label{eq:fg}
		[x,y]=x\cdot f(y)+g(x\cdot y),\;\forall x,y\in A,
	\end{equation}
	which is called a \textbf{{special left-Alia algebra}} with respect to $(A,\cdot,f,g)$.
	If $D$ is a twisted derivation of $(A,\cdot)$ with respect to $R$, then we see that $(A,[\cdot,\cdot]_{R})$ satisfies \eqref{eq:fg} for $$f=2D, g=-D.$$
	
	Hence, $(A,[\cdot,\cdot]_{R})$ is left-Alia.
\end{rmk}

\subsection{Examples of Left-Alia Algebras}

\vspace{6pt}
\begin{ex}
	Let $R$ be a reflection defined by $R(x_1)=x_2,R(x_2)=x_1,R(x_3)=x_3$ on three-dimensional vector space $V^*$ with a basis $\{x_1,x_2,x_3\}$. On the coordinate ring \mbox{$S=\mathbb{K}[x_1,x_2,x_3]$}  of $V$, $R$ can be also denoted an extension of $R$ satisfying \mbox{$R(fg)=R(f)R(g)$} and $R(k_1 f+k_2h)=k_1R(f)+k_2R(h)$. Let $D$ be the twisted derivation on $S$ induced by the reflection $R$. It follows from Theorem~\ref{main12} that $R(f)=f-D(f)(x_1-x_2)$. Take two polynomials, $f=\underset{i}{\sum}k_i f_i,g=\underset{j}{\sum}h_j g_j, k_i,h_j\in\mathbb{K}$, in $S$, where $f_i,g_j$ are monomials, $f_i=x_1^{n_{i,1}}x_2^{n_{i,2}}x_3^{n_{i,3}}, g_j=x_1^{m_{j,1}}x_2^{m_{j,2}}x_3^{m_{j,3}}$. We have
	\begin{align*}
		D(x_1)=&1, D(x_2)=-1, D(x_3)=0,\\
		D(x_1^{n_1})=&x_1^{n_1-1}+x_1^{n_1-2}x_2+\cdots+x_2^{n_1-1},\\
		D(x_2^{n_2})=&-x_1^{n_2-1}-x_1^{n_2-1}x_2-\cdots-x_2^{n_2-1},\\
		D(x_3^{n_3})=&0,
	\end{align*}
and
\begin{align*}
		D(\underset{i}{\sum} k_ix_1^{n_{i,1}}x_2^{n_{i,2}}x_3^{n_{i,3}})=&\underset{i}{\sum}k_i(D(x_1^{n_{i,1}}x_2^{n_{i,2}})x_3^{n_{i,3}}+R(x_1^{n_{i,1}}x_2^{n_{i,2}})D(x_3^{n_{i,3}}))\\
		=&\underset{i}{\sum} k_i(x_1^{n_{i,1}}D(x_2^{n_{i,2}})+R(x_2^{n_{i,2}})D(x_1^{n_{i,1}}))x_3^{n_{i,3}}\\
		=&\underset{i}{\sum} k_i(x_1^{n_{i,1}+n_{i,2}-1}+x_1^{n_{i,1}+n_{i,2}-2}x_2+\cdots+x_1^{n_{i,2}}x_2^{n_{i,1}-1}\\
		&-x_1^{n_{i,1}+n_{i,2}-1}-\cdots-x_1^{n_1}x_2^{n_2-1})x_3^{n_{i,3}}.
	\end{align*}
	
Let $[\cdot,\cdot]_R:S\times S\rightarrow S$ be the bilinear map defined in Theorem~\ref{thm:R-D}. Then,
	\begin{align*}
		[f,g]_R=&\underset{i,j}{\sum}k_ih_j [f_i,g_j]_R\\
		=&\underset{i,j}{\sum}k_ih_j (f_iD(g_j)-R(g_j)D(f_i))\\
		=&\underset{i,j}{\sum}k_ih_j (x_1^{n_{i,1}+m_{j,1}+m_{j,2}-1}x_2^{n_{i,2}}+\cdots+x_1^{n_{i,1}+m_{j,2}}x_2^{n_{i,2}+m_{j,1}-1}\\
		&-x_1^{n_{i,1}+m_{j,1}+m_{j,2}-1}x_2^{n_{i,2}}-\cdots-x_1^{m_{j,1}+n_{i,1}}x_2^{n_{i,2}+m_{j,2}-1}\\
		&-x_1^{m_{j,2}+n_{i,1}+n_{i,2}-1}x_2^{m_{j,1}}-\cdots-x_1^{n_{i,2}+m_{j,2}}x_2^{n_{i,1}+m_{j,1}-1}\\
		&+x_1^{m_{j,2}+n_{i,1}+n_{i,2}-1}x_2^{m_{j,1}}+\cdots+x_1^{n_{i,1}+m_{j,2}}x_2^{n_{i,2}+m_{j,1}-1})x_3^{m_{j,3}+n_{i,3}}.
	\end{align*}
	
Since $(S,\cdot)$ is a commutative associative algebra, by Theorem~\ref{thm:R-D} $(S,[\cdot,\cdot]_R)$ is a left-Alia algebra.
\end{ex}

\begin{pro}
	Let $(A, [\cdot,\cdot])$  be an $n$-dimensional ($n\geq2$) left-Alia algebra and  $\left\{e_{1}, \cdots, e_{n}\right\}$  be a basis of  $A$. For all positive integers~$1\leq i, j, t\leq n$ and structural constants $C_{i j}^{t}\in \mathbb{C}$,  set
	\begin{equation}\label{eq:muilt}
		[e_{i}, e_{j}]=\sum_{t=1}^{n} C_{i j}^{t} e_{t}.
	\end{equation}
	
	Then, $(A, [\cdot,\cdot])$ is a left-Alia algebra if and only if the structural constants $C_{i j}^{t}$ satisfy the \linebreak  following equation:
	\begin{equation}\label{structural-constants}
		\sum_{k,m=1}^{n}\left((C_{ij}^k-C_{ji}^k)C_{kl}^m+(C_{jl}^k-C_{lj}^k)C_{ki}^m+(C_{li}^k-C_{il}^k)C_{kj}^m\right)=0,\;\forall 1\leq i, j,l\leq n.
	\end{equation}
\end{pro}
\begin{proof}
	By \eqref{eq:0-Alia}, for all $e_{i}, e_{j},e_{l} \in \left\{e_{1}, \cdots, e_{n}\right\}$, %$1\leq i, j,l\leq n$,
	\begin{align}\label{eq:c-J}
		[[e_i, e_j]-[e_j,e_i], e_l]+[[e_j, e_l]-[e_l,e_j], e_i]+[[e_l, e_i]-[e_i,e_l], e_j]=0.
	\end{align}
	
Set
	\begin{equation*}
		[e_{i},e_j]=\sum_{k=1}^{n} C_{ij}^k e_{k},~[e_{j},e_l]=\sum_{k=1}^{n} C_{jl}^{k} e_{l},~ [e_{l} \cdot e_{i}]=\sum_{k=1}^{n} C_{li}^k e_{k},~~~~ C_{ij}^k,C_{jl}^{k},C_{li}^k\in \mathbb{C}.
	\end{equation*}
	
Therefore, Equation~\eqref{structural-constants} holds.
\end{proof}

As a direct consequence, we obtain the following:
\begin{pro}
	Let $A$ be a two-dimensional vector space
	over the complex field $\mathbb{C}$ with a basis $\{e_1, e_2\}$. Then, for any bilinear map $[\cdot,\cdot]$
	on $A$, $(A,[\cdot,\cdot])$ is a left-Alia algebra.
\end{pro}
Next, we give some example of three-dimensional left-Alia algebras.

\begin{ex}\label{ex:3-left}
	Let $A$ be a three-dimensional vector space
	over the complex field $\mathbb{C}$ with a basis $\{e_1, e_2, e_3\}$. Define a bilinear map $[\cdot,\cdot]:A\times A\rightarrow A$ by
	\begin{align*}
		&[e_1,e_2]=e_1,~~[e_1,e_3]=e_1,~~[e_2,e_1]=e_2,~~[e_3,e_1]=e_3,\\
		&[e_1,e_1]=[e_2,e_2]=[e_3,e_3]=[e_2,e_3]=[e_3,e_2]=e_1+e_2+e_3.
	\end{align*}
	
	Then, $(A, [\cdot,\cdot])$ is a three-dimensional left-Alia algebra.
\end{ex}

\begin{rmk}
	A right-Leibniz algebra~\cite{Loday} is a vector space $A$ together with a bilinear operation  $[\cdot,\cdot]:A\otimes A\rightarrow A$ satisfying
	\begin{equation*}
		[[x, y], z]=[[x, z],y]+[x, [y, z]],\;\forall~ x,y,z\in A.
	\end{equation*}
	
	Then, we have
	\begin{equation*}
		[x,[y, z]]+[y,[z, x]]+[z,[x, y]]=[[x,y]-[y,x],z]+[[y,z]-[z,y],x]+[[z,x]-[x,z],y].
	\end{equation*}
	
	Therefore, if  a right-Leibniz algebra satisfies
	\begin{equation*}
		[x,[y, z]]+[y,[z, x]]+[z,[x, y]]=0,
	\end{equation*}
	then $(A,[\cdot,\cdot])$ is a left-Alia algebra.
\end{rmk}

\subsection{From Left-Alia Algebras to Anti-Pre-Lie Algebras}

\vspace{6pt}
\begin{defi}[\cite{LB2022}]
	Let $A$ be a vector space with a bilinear map $\cdot: A\times A\rightarrow A$. $(A,\cdot)$
	is called an \textbf{{anti-pre-Lie algebra}} if the following equations are satisfied:
	\begin{align}
		&x\cdot(y\cdot z)-y\cdot (x\cdot z)=[y,x]\cdot z,\label{eq:(5)}\\
		&[x,y]\cdot z+[y,z]\cdot x+[z,x]\cdot y=0,
	\end{align}
	where
	\begin{equation}
		[x, y] = x\cdot y -y \cdot x,
	\end{equation}
	for all $x, y, z \in A$.
\end{defi}

\begin{rmk}
	Let $(A,[\cdot,\cdot])$ be a left-Alia algebra. If $[\cdot,\cdot]:A\times A\rightarrow A$ satisfies
	\begin{equation}\label{eq:anti-pre-Lie}
		[x,[y, z]]-[y,[x,z]]=[[y,x],z]-[[x,y],z], \;\forall x,y,z\in A,
	\end{equation}
	then $(A,[\cdot,\cdot])$ is  an anti-pre-Lie algebra.
\end{rmk}

\subsection{From Left-Alia Algebras to Lie Triple Systems}

Lie triple systems originated from Cartan's studies on the Riemannian
geometry of totally geodesic submanifolds~\cite{Helgason}, which can be constructed using twisted derivations and left-Alia algebras.

\begin{defi}[\cite{Hodge}]
A \textbf{{Lie triple system}} is a vector space $A$ together with a trilinear operation $[\cdot,\cdot,\cdot]:A \times A \times A \rightarrow A$ such that the following three equations are  satisfied, for all $ x,~y,~z,~a,~b$ in $A$:
	\begin{align}
		[x,x,y] & =0, \label{eq:skew-symm}\\
		[x,y,z]+[y,z,x]+[z,x,y] & =0 \label{eq:cyc}
	\end{align}
	and
	\begin{equation}\label{eq:F-J}
		[a,b,[x,y,z]]=[[a,b,x],y,z]+[x,[a,b,y],z]+[x,y,[a,b,z]].
	\end{equation}
\end{defi}

\begin{pro}
	Let $(A,\cdot)$ be a commutative associative algebra and $D$ be a twisted derivation. \linebreak  Define
	the bilinear map $[\cdot,\cdot]_R:A\times A\rightarrow A$ by \eqref{eq:R-der}.
	And define the trilinear \linebreak  map $[\cdot,\cdot,\cdot]_R:A\times A\times A\rightarrow A$ by
	$$
	[x,y,z]_R:=\frac{1}{2}[[x,y]_R-[y,x]_R,z]_R,\;\forall x,y,z\in A.
	$$
	
	If $[\cdot,\cdot,\cdot]_R$ satisfies \eqref{eq:F-J}, then $(A,[\cdot,\cdot,\cdot]_R)$ is a Lie triple system.
\end{pro}
\begin{proof}
	For all $x,y\in A$, it is obvious that $[x,x,y]_R=0$. By the proof of Theorem~\ref{thm:R-D}, Equation~\eqref{eq:cyc} holds. If, in addition, $[\cdot,\cdot,\cdot]_R$ satisfies \eqref{eq:F-J}, then $(A,[\cdot,\cdot,\cdot]_R)$ is a Lie triple~system.
\end{proof}

\begin{rmk}
	Let $(A,[\cdot,\cdot])$ be a left-Alia algebra. For all $x,y,z\in A$, set a trilinear \linebreak  map $[\cdot,\cdot,\cdot]:A \times A \times A \rightarrow A$ by $[x,y,z]=[[x, y]-[y,x], z]$. If $[\cdot,\cdot,\cdot]$ satisfies \eqref{eq:F-J},
	then $(A,[\cdot,\cdot,\cdot])$ is  a Lie triple system.
\end{rmk}

\subsection{Representations and Matched Pairs of left-Alia Algebras}

\vspace{6pt}
\begin{defi}
	A \textbf{{representation}} of a left-Alia algebra $(A,[\cdot,\cdot])$ is a triple $(l,r,V)$, where $V$ is a vector space and $l,r:A\rightarrow\mathrm{End}(V)$ are linear maps such that the following equation holds:
	\begin{equation}\label{eq:rep 0-Alia}
		l([x,y])-l([y, x])= r(x)r(y)v-r(y)r(x)+r(y)l(x)v-r(x)l(y),\;\forall x,y\in A, v\in V.
	\end{equation}
	
	Two representations, $(l,r,V)$ and $(l',r',V')$, of a left-Alia algebra $(A, [\cdot,\cdot] )$ are called \textbf{{equivalent}} if there is a linear isomorphism $\phi:V\rightarrow V'$ such that
	\begin{equation}
		\phi\big( l(x)v\big)=l'(x)\phi(v),\; \phi\big( r(x)v\big)=r'(x)\phi(v),\;\forall x\in A, v\in V.
	\end{equation}
\end{defi}

\begin{ex}
	Let $(\rho,V)$ be a representation of a Lie algebra $(\mathfrak{g},[\cdot,\cdot])$, that is, $\rho:\mathfrak{g}\rightarrow \mathrm{End}(V)$ is a linear map such that
	\begin{equation*}
		\rho([x,y])v=\rho(x)\rho(y)v-\rho(y)\rho(x)v,\;\forall x,y\in\mathfrak{g}, v\in V.
	\end{equation*}
	
	Then, both $(\rho,-\rho,V)$ and  $(\rho,2\rho,V)$ satisfy \eqref{eq:rep 0-Alia} and, hence, are representations of  $(\mathfrak{g},[\cdot,\cdot])$ as a left-Alia algebra.
\end{ex}

%\delete{
%	\begin{Remark}
%		Let $(l,r,V)$ be a representation of left-Alia algebra $(A,[\cdot,\cdot])$. If the multiplication $[\cdot,\cdot]$ is skew-symmetric and  satisfies  $r(x)l(y)=-r(x)r(y)$, then
%		$$
%		l([x,y])v-r(x)r(y)v+r(y)r(x)v=0, \;\forall x,y\in A, v\in V.
%		$$
%	\end{Remark}
%	\begin{Remark}
%		When $(A,[\cdot,\cdot])$ is a Lie algebra $r=2l$, $l$ and $r$ are representations of the Lie algebra $A$.
%\end{Remark}}

\begin{pro}\label{pro:semi}
	Let $(A,[\cdot,\cdot])$ be a left-Alia algebra, $V$ be a vector space and $l,r:A\rightarrow\mathrm{End}(V)$ be linear maps.
	Then, $(l,r,V)$ is a representation of $(A,[\cdot,\cdot])$ if and only if there is a left-Alia algebra on the direct sum $d=A\oplus V$ of vector spaces (\textbf{{the semi-direct product}}) given by
	\begin{equation}\label{eq:semi-d}
		[x+u,y+v]_d= [x, y]+l(x)v+r(y)u ,\;\forall x,y\in A, u,v\in V.
	\end{equation}
	
	In this case, we denote $(A\oplus V,[\cdot,\cdot]_d)=A\ltimes_{l,r}V$.
\end{pro}
\begin{proof}
	This is the special case of matched pairs of left-Alia algebras where $B=V$ is equipped with the zero multiplication in Proposition~\ref{pro:322}.
\end{proof}

For a vector space $A$ with a bilinear map $[\cdot,\cdot]:A\times A\rightarrow A$, we set linear maps $\mathcal{L}_{[\cdot,\cdot]},\mathcal{R}_{[\cdot,\cdot]}:A\rightarrow\mathrm{End}(A)$ using
\begin{equation*}
	\mathcal{L}_{[\cdot,\cdot]}(x)y=[x, y]=\mathcal{R}_{[\cdot,\cdot]}(y)x,\;\forall x,y\in A.
\end{equation*}
\begin{ex}
	Let $(A,[\cdot,\cdot])$ be a left-Alia algebra.
	Then, $(\mathcal{L}_{[\cdot,\cdot]},\mathcal{R}_{[\cdot,\cdot]},A)$ is a representation of $(A,[\cdot,\cdot])$, which is called an \textbf{{adjoint representation}}.
	In particular, for a Lie algebra $(\mathfrak{g},[\cdot,\cdot])$ with the adjoint representation $\mathrm{ad}:\mathfrak{g}\rightarrow\mathrm{End}(\mathfrak{g})$ given by $\mathrm{ad}(x)y=[x,y],\;\forall x,y\in\mathfrak{g}$,
	\begin{equation*}
		(\mathcal{L}_{[\cdot,\cdot]},\mathcal{R}_{[\cdot,\cdot]},\mathfrak{g})=(\mathrm{ad},-\mathrm{ad},\mathfrak{g})
	\end{equation*}
	is a representation of $(\mathfrak{g},[\cdot,\cdot])$ as a left-Alia algebra.
\end{ex}

%\delete{
%	\begin{Remark}
%		An adjoint representation of a Lie algebra $\g$ is a linear map $\ad:\g\rightarrow \mathrm{End}(\g)$ such that for all $x,y\in \g$, $\ad(x)(y):=[x,y]$. If the bilinear map  in left-Alia algebra $(A,[\cdot,\cdot])$ is skew-symmetric, then $\mathcal{L}_{[\cdot,\cdot]},\mathcal{R}_{[\cdot,\cdot]}:A\rightarrow\mathrm{End}(A)$ are adjoint representations of Lie algebras.
%\end{Remark}}

Let $A$ and $V$ be vector spaces. For a linear map $l :A\rightarrow\mathrm{End}(V)$, we set a linear map $l^{*}:A\rightarrow\mathrm{End}(V^{*})$ using
\begin{equation*}
	\langle l^{*}(x)u^{*},v\rangle=-\langle u^{*},l(x)v\rangle,\;\forall x\in A, u^{*}\in V^{*}, v\in V.
\end{equation*}

\begin{pro}\label{pro:dual rep}
	Let $(l,r,V)$ be a representation of a left-Alia algebra $(A,[\cdot,\cdot])$.
	Then, \linebreak  $( l^{*}, l^{*}-r^{*}, V^{*} )$ is also a representation of $(A,[\cdot,\cdot])$.
	In particular, $( \mathcal{L}^{*}_{[\cdot,\cdot]},\mathcal{L}^{*}_{[\cdot,\cdot]}-\mathcal{R}^{*}_{[\cdot,\cdot]}, A^{*} )$ is a representation of $(A,[\cdot,\cdot])$, which is called the \textbf{{coadjoint representation}}.
\end{pro}
\begin{proof}
	Let $x,y\in A, u^{*}\in V^{*},v\in V$. Then, we have
	\begin{eqnarray*}
		&&\langle\big( l^{*}[x, y]- l^{*} [y, x]+(l^{*}-r^{*})(x)l^{*}(y)
		-(l^{*}-r^{*})(x)(l^{*}-r^{*})(y)\\
		&&+(l^{*}-r^{*})(y)(l^{*}-r^{*})(x)-(l^{*}-r^{*})(y)l^{*}(x)\big)u^{*},v\rangle\\
		&&=\langle (l^{*}[x, y]- l^{*} [y, x]+(l^{*}-r^{*})(x)r^{*}(y)-(l^{*}-r^{*})(y)r^{*}(x)\big)u^{*},v\rangle\\
		&&=\langle u^{*}, \big(l[y, x]-l[x, y]+r(y)(l-r)(x)-r(x)(l-r)(y)\big)v\rangle\\
		&&\overset{\eqref{eq:rep 0-Alia}}{=}0.
	\end{eqnarray*}
	
	{Hence, the conclusion follows.} %MDPI: This paragraph is duplicated 6 times in this paper. Please check and revise if necessary.
\end{proof}

\begin{ex}
	Let $(\mathfrak{g},[\cdot,\cdot])$ be a Lie algebra.
	Then, the coadjoint representation of $(\mathfrak{g},[\cdot,\cdot])$ as a left-Alia algebra is
	\begin{equation*}
		(\mathrm{ad}^{*},\mathrm{ad}^{*}-(-\mathrm{ad}^{*}),\mathfrak{g}^{*})=( \mathrm{ad}^{*},2\mathrm{ad}^{*},\mathfrak{g}^{*}  ).
	\end{equation*}
	
Hence, there is a left-Alia algebra structure $\mathfrak{g}\ltimes_{\mathrm{ad}^{*}, 2\mathrm{ad}^{*}}\mathfrak{g}^{*}$ on the direct sum $\mathfrak{g}\oplus \mathfrak{g}^{*}$ of vector~spaces.
\end{ex}

%\delete{
%	\begin{Example}\label{ex:co-ad}
%		Let $(A,[\cdot,\cdot])$ be a 3-dimensional left-Alia algebra with a basis $\{e_1, e_2, e_3\}$ given by Example~\ref{ex:3-left}, and $\{e_1^*, e_2^*, e_3^*\}$ be the dual basis of $A^*$.  Then  $( \mathcal{L}^{*}_{[\cdot,\cdot]},\mathcal{L}^{*}_{[\cdot,\cdot]}-\mathcal{R}^{*}_{[\cdot,\cdot]}, A^{*} )$ is the coadjoint representation, where
%		\begin{align*}
%			&\mathcal{L}^{*}_{[\cdot,\cdot]}(e_1)(e_2^*)=\mathcal{L}^{*}_{[\cdot,\cdot]}(e_1)(e_3^*)=-e_1^*,\\
%			&\mathcal{L}^{*}_{[\cdot,\cdot]}(e_1)(e_1^*)=\mathcal{L}^{*}_{[\cdot,\cdot]}(e_2)(e_2^*)
%			=\mathcal{L}^{*}_{[\cdot,\cdot]}(e_3)(e_3^*)=-e_1^*-e_2^*-e_3^*,\\
%			&\mathcal{L}^{*}_{[\cdot,\cdot]}(e_2)(e_1^*)=\mathcal{L}^{*}_{[\cdot,\cdot]}(e_2)(e_3^*)
%			=\mathcal{L}^{*}_{[\cdot,\cdot]}(e_3)(e_1^*)=\mathcal{L}^{*}_{[\cdot,\cdot]}(e_3)(e_2^*)=-e_2^*-e_3^*,
%		\end{align*}
%		and
%		\begin{align*}
%			&\mathcal{R}^{*}_{[\cdot,\cdot]}(e_2)(e_1^*)=\mathcal{R}^{*}_{[\cdot,\cdot]}(e_3)(e_1^*)
%			=-e_1^*-e_2^*-e_3^*,\\
%			&\mathcal{R}^{*}_{[\cdot,\cdot]}(e_1)(e_1^*)=-e_1^*,~ \mathcal{R}^{*}_{[\cdot,\cdot]}(e_1)(e_2^*)=-e_1^*-e_2^*,~ \mathcal{R}^{*}_{[\cdot,\cdot]}(e_1)(e_3^*)=-e_1^*-e_3^*\\
%			&\mathcal{R}^{*}_{[\cdot,\cdot]}(e_2)(e_2^*)=\mathcal{R}^{*}_{[\cdot,\cdot]}(e_3)(e_3^*)
%			=\mathcal{R}^{*}_{[\cdot,\cdot]}(e_2)(e_3^*)
%			=\mathcal{R}^{*}_{[\cdot,\cdot]}(e_3)(e_2^*)=-e_2^*-e_3^*.
%		\end{align*}
%\end{Example}}

\begin{rmk}
	In~\cite{Dzh09}, there is also the notion of a \textbf{{right-Alia algebra}}, defined as a vector space $A$ together with a bilinear map $[\cdot,\cdot]':A\times A\rightarrow A$ satisfying
\begin{equation}\label{eq:right Alia}
		[x,[y,z]']'+[y,[z,x]']'+[z,[x,y]']'=[x,[z,y]']'+[y,[x,z]']'+[z,[y, x]']',\;\forall x,y,z\in A.
	\end{equation}
	It is clear that $(A,[\cdot,\cdot]')$ is a right-Alia algebra if and only if the opposite algebra $(A,[\cdot,\cdot])$ of $(A,[\cdot,\cdot]')$, given by
	$[x,y]=[y,x]'$, is a left-Alia algebra.
	Thus, our study on left-Alia algebras can straightforwardly generalize a parallel study on right-Alia algebras.
	Consequently, if $(l,r,V)$ is a representation of a right-Alia algebra $(A,[\cdot,\cdot]')$, then $(r^{*}-l^{*},r^{*}, V^{*})$ is also a representation of $(A,[\cdot,\cdot]')$. Recall ~\cite{LB2023} that if $(l,r,V)$ is a representation of an anti-pre-Lie  algebra $(A,[\cdot,\cdot]_{\mathrm{anti}})$, then $(r^{*}-l^{*},r^{*}, V^{*})$ is also a representation of $(A,[\cdot,\cdot]_{\mathrm{anti}})$. Moreover, admissible Novikov algebras~\cite{LB2022}  are a subclass of anti-pre-Lie algebras.
	If $(l,r,V)$ is a representation of an admissible Novikov algebra $(A,[\cdot,\cdot]_{admissible\ Novikov})$, then $(r^{*}-l^{*},r^{*}, V^{*})$ is also a representation of the admissible Novikov algebra $(A,[\cdot,\cdot]_{admissible\ Novikov})$. Therefore, we have the following algebras which preserve the form $(r^{*}-l^{*},r^{*}, V^{*})$ of representations on the dual spaces:
	
	\vspace{6pt}
	$\{$right-Alia algebras$\}\supset\{$anti-pre-Lie algebras$\}\supset\{$admissible Novikov algebras$\}.$
\end{rmk}

\delete{
	\begin{Remark}
		Let $\g$ be a Lie algebra and $\g^*$ be its dual vector space. If a linear map $\ad^{*}:\g\rightarrow\mathrm{End}(\g^*)$ satisfying
		\begin{equation}\label{eq:ad}
			\langle \ad^{*}(x)\xi,y\rangle=-\langle \xi,\ad(x)y\rangle,\;\forall x,y\in A, \xi\in \g^{*}.
		\end{equation}
		Then $\ad^*$ is a representation of $\g$ in $\g^*$, which is called the coadjoint
		representation of $\g$. In Proposition~\ref{pro:dual rep}, if the left-Alia algebra $(A,[\cdot,\cdot])$ is skew-symmetric, that is for all $x\in A$ $(\mathcal{L}_{[\cdot,\cdot]}-\mathcal{R}_{[\cdot,\cdot]})(x)=2\mathcal{L}_{[\cdot,\cdot]}(x)$, then $\mathcal{L}_{[\cdot,\cdot]}-\mathcal{R}_{[\cdot,\cdot]}$ is a adjoint
		representation of Lie algebra. Furthermore, by \eqref{eq:ad}, $\mathcal{L}_{[\cdot,\cdot]}^*$ and
		$\mathcal{L}^{*}_{[\cdot,\cdot]}-\mathcal{R}^{*}_{[\cdot,\cdot]}$ are coadjoint
		representations of Lie algebras.
\end{Remark}}

Now, we introduce the notion of matched pairs of left-Alia algebras.

\begin{defi}\label{match}
	Let $(A,[\cdot,\cdot]_A )$ and $(B,[\cdot,\cdot]_B )$ be left-Alia algebras and $l_{A},r_{A}:A\rightarrow\mathrm{End}(B)$ and $l_{B},r_{B}:B\rightarrow\mathrm{End}(A)$ be linear maps.
	If there is a left-Alia algebra structure $[\cdot,\cdot]_{A\oplus B}$ on the direct sum $A\oplus B$ of vector spaces given by
	\begin{equation*}
		[x+a,y+b]_{A\oplus B}=[x,y]_A+l_{B}(a)y+r_{B}(b)x+[a,b]_B+l_{A}(x)b+r_{A}(y)a,\;\forall x,y\in A, a,b\in B,
	\end{equation*}
	then we say $\big( (A,[\cdot,\cdot]_A ),(B,[\cdot,\cdot]_B ), l_{A}, r_{A}, l_{B}, r_{B} \big)$ is a \textbf{{matched pair of left-Alia algebras}}.
\end{defi}

\begin{pro}\label{pro:322}
	Let $(A,[\cdot,\cdot]_A )$ and $(B,[\cdot,\cdot]_B )$ be left-Alia algebras and $l_{A},r_{A}:A\rightarrow\mathrm{End}(B)$ and $l_{B},r_{B}:B\rightarrow\mathrm{End}(A)$ be linear maps. Then, $\big( (A,[\cdot,\cdot]_A ),(B,[\cdot,\cdot]_B ), l_{A}, r_{A}, l_{B}, r_{B} \big)$ is a   matched pair of left-Alia algebras if and only if the triple $(l_{A}, r_{A}, B)$ is a representation of $(A,[\cdot,\cdot]_A )$, the triple
	$(l_{B}, r_{B}, A)$ is a representation of $(B,[\cdot,\cdot]_B )$ and the following equations hold:
	\begin{align}
		r_{B}(a)([x,y]_A-[y,x]_A)&=(l_{B}-r_{B})(a)[y,x]_A+(r_{B}-l_{B})(a)[x, y]_A\nonumber\\
		&+l_{B}\big((r_{A}-l_{A})(y)a\big)x+l_{B}\big((l_{A}-r_{A})(x)a\big)y,\label{eq:defi:matched pairs1}\\
		r_{A}(x)([a,b]_B-[b,a]_B)&=(l_{A}-r_{A})(x)[b,a]_{B}+(r_{A}-l_{A})(x)[a, b]_{B}\nonumber\\
		&+l_{A}\big((r_{B}-l_{B})(b)x\big)a+l_{A}\big((l_{B}-r_{B})(a)x\big)b,\label{eq:defi:matched pairs2}
	\end{align}
	
	for all $x,y\in A, a,b\in B$.
\end{pro}
\begin{proof}
	The proof follows from a straightforward computation. %Please check intended meaning has been retained
\end{proof}

\subsection{Quadratic Left-Alia Algebras}

\vspace{6pt}
\begin{defi}\label{qua}
	A \textbf{{quadratic left-Alia algebra}} is a triple $(A,[\cdot,\cdot],\mathcal{B})$, where $(A,[\cdot,\cdot])$ is a left-Alia algebra and $\mathcal{B}$ is a nondegenerate symmetric bilinear form on $A$ which is invariant in the sense that
	
	\vspace{-6pt}\begin{equation}\label{eq:quad}
		\mathcal{B}([x, y],z)=\mathcal{B}(x,[z, y]-[y, z]),\;\forall x,y,z\in A.
	\end{equation}
\end{defi}
\begin{rmk}
	Since $\mathcal{B}$ is symmetric, it follows from Definition~\ref{qua} that
	\begin{equation}\label{eq:inv}
		\mathcal{B}([x, y],z)+\mathcal{B}(y,[x, z])=0,\;\forall x,y,z\in A.
	\end{equation}
\end{rmk}

\begin{lem}
	Let $(A,[\cdot,\cdot],\mathcal{B})$ be a quadratic left-Alia algebra.
	Then, $(\mathcal{L}_{[\cdot,\cdot]},\mathcal{R}_{[\cdot,\cdot]},A)$ and \linebreak $( \mathcal{L}^{*}_{[\cdot,\cdot]},\mathcal{L}^{*}_{[\cdot,\cdot]}-\mathcal{R}^{*}_{[\cdot,\cdot]}, A^{*} )$ are equivalent as representations of $(A,[\cdot,\cdot])$.
\end{lem}
\begin{proof}
	We set a linear isomorphism $\mathcal{B}^{\natural}:A\rightarrow A^{*}$ using
	
	\vspace{-6pt}\begin{equation}\label{eq:bn}
		\langle\mathcal{B}^{\natural}(x),y\rangle=\mathcal{B}(x,y).
	\end{equation}
	
	Then, by \eqref{eq:inv} we have
	\begin{equation*}
		\langle\mathcal{B}^{\natural}\big(\mathcal{L}_{[\cdot,\cdot]}(x)y\big),z\rangle=\mathcal{B}([x, y],z)=-\mathcal{B}(y,[x,z])=-\langle\mathcal{B}^{\natural}(y),[x, z]\rangle=\langle\mathcal{L}^{*}_{[\cdot,\cdot]}(x)\mathcal{B}^{\natural}(y),z\rangle,
	\end{equation*}
	that is, $\mathcal{B}^{\natural}\big(\mathcal{L}_{[\cdot,\cdot]}(x)y\big)=\mathcal{L}^{*}_{[\cdot,\cdot]}(x)\mathcal{B}^{\natural}(y)$.
	Similarly, by \eqref{eq:quad}, we have $\mathcal{B}^{\natural}\big(\mathcal{R}_{[\cdot,\cdot]}(x)y\big)= \linebreak  (\mathcal{L}^{*}_{[\cdot,\cdot]}-\mathcal{R}^{*}_{[\cdot,\cdot]})(x)\mathcal{B}^{\natural}(y)$. Hence, the conclusion follows.
\end{proof}

\begin{pro}\label{pro:2.8}
	Let $(A,\cdot)$ be a commutative associative algebra and $f:A\rightarrow A$ be a linear map.
	Let $\mathcal{B}$ be a nondegenerate symmetric invariant bilinear form on $(A,\cdot)$
	and $\hat{f}:A\rightarrow A$ be the adjoint map of $f$ with respect to $\mathcal{B}$, given by
	
	\vspace{-6pt}\begin{equation*}
		\mathcal{B}\big(\hat{f}(x),y\big)=\mathcal{B}\big(x,f(y)\big),\;\forall x,y\in A.
	\end{equation*}
	
	Then, there is a quadratic left-Alia algebra $(A,[\cdot,\cdot],\mathcal{B})$, where $(A,[\cdot,\cdot])$ is the special left-Alia algebra with respect to $(A,\cdot,f,-\hat f)$, that is,
	
	\vspace{-6pt}\begin{equation}\label{eq:hat}
		[x, y]=x\cdot f(y)-\hat{f}(x\cdot y).
	\end{equation}
\end{pro}
\begin{proof}
	For all $x,y,z\in A$, we have
	
	\vspace{-12pt}\begin{eqnarray*}
		\mathcal{B}([x,y],z)&=&\mathcal{B}\big(x\cdot f(y)-\hat{f}(x\cdot y),z\big)\\
		&=&\mathcal{B}\big( x, z\cdot f(y) -y\cdot f(z)\big)\\
		&=&\mathcal{B}\Big (x,   z\cdot f(y)-\hat{f}(z\cdot y)-\big( y\cdot f(z)-\hat{f}(y\cdot z)  \big)\Big)\\
		&=&\mathcal{B}(x,[z, y]-[y,z]).
	\end{eqnarray*}
	
	Hence, the conclusion follows.
\end{proof}

%\delete{
%	\begin{Remark}
%		The left-Alia algebra defined by (\ref{eq:hat}) coincides with the special left-Alia  in
%		\cite[Theorem 6.2]{Dzh09}. In particular, when $f=2D$ and $\hat{f}=-D$ for a twisted derivation $D$, (\ref{eq:hat}) is precisely the bilinear map defined in Theorem~\ref{thm:R-D}.
%\end{Remark}}

\begin{ex}\label{ex:2.9}
	Let $(A,[\cdot,\cdot])$ be a left-Alia algebra and $(\mathcal{L}_{[\cdot,\cdot]},\mathcal{R}_{[\cdot,\cdot]},A)$ be the adjoint representation of $(A,[\cdot,\cdot])$.
	By Propositions~\ref{pro:semi} and~\ref{pro:dual rep}, there is a left-Alia algebra $A\ltimes_{\mathcal{L}^{*}_{[\cdot,\cdot]}, \mathcal{L}^{*}_{[\cdot,\cdot]}-\mathcal{R}^{*}_{[\cdot,\cdot]}}A^{*}$ on $d=A\oplus A^*$, given by \eqref{eq:semi-d}.
	There is a natural nondegenerate symmetric bilinear form $\mathcal{B}_{d}$ on $A\oplus A^{*}$, given by
	\begin{equation}\label{eq:Bd}
		\mathcal{B}_{d}(x+a^{*},y+b^{*})=\langle x, b^{*}\rangle+\langle a^{*},y\rangle,\;\forall x,y\in A, a^{*},b^{*}\in A^{*}.
	\end{equation}
	
	For all $x,y,z\in A, a^{*},b^{*},c^{*}\in A^{*}$, we have
	\begin{eqnarray*}
		\mathcal{B}_{d} ([ x+a^{*} , y+b^{*} ]_d,z+c^{*})
		&=&\mathcal{B}_{d}\big([x, y]+\mathcal{L}^{*}_{[\cdot,\cdot]}(x)b^{*}+(\mathcal{L}^{*}_{[\cdot,\cdot]}-\mathcal{R}^{*}_{[\cdot,\cdot]})(y)a^{*},z+c^{*}\big)\\
		&=&\langle [x, y],c^{*}\rangle+\langle \mathcal{L}^{*}_{[\cdot,\cdot]}(x)b^{*}+(\mathcal{L}^{*}_{[\cdot,\cdot]}-\mathcal{R}^{*}_{[\cdot,\cdot]})(y)a^{*},z\rangle\\
		&=&\langle [x,y],c^{*}\rangle-\langle [x,z], b^{*}\rangle+\langle a^{*},[z,y]-[y,z]\rangle,\\
		\mathcal{B}_{d}\big(x+a^{*},[z+c^{*},x+b^{*}]_d\big)&=&\langle [z,y], a^{*}\rangle-\langle [z,x], b^{*}\rangle+\langle c^{*}, [y,x]-[x, y]\rangle,\\
		\mathcal{B}_{d}\big(x+a^{*},[y+b^{*},z+c^{*}]_d\big)&=&\langle [y,z], a^{*}\rangle-\langle [y,x], c^{*}\rangle+\langle b^{*}, [z,x]-[x,z]\rangle.
	\end{eqnarray*}
	
	Hence, we have
	\begin{equation*}
		\mathcal{B}_{d}\big( [x+a^{*},y+b^{*}]_d,z+c^{*} \big)=\mathcal{B}_{d}\big( x+a^{*}, [z+c^{*},y+b^{*}]_d- [y+b^{*},z+c^{*}]_d\big),
	\end{equation*}
	and, thus, $(A\ltimes_{\mathcal{L}^{*}_{[\cdot,\cdot]}, \mathcal{L}^{*}_{[\cdot,\cdot]}-\mathcal{R}^{*}_{[\cdot,\cdot]}}A^{*},\mathcal{B}_{d})$ is a quadratic left-Alia algebra.
\end{ex}

\begin{rmk}
	By Example~\ref{ex:2.9}, an arbitrary Lie algebra $(\mathfrak{g},[\cdot,\cdot])$ renders a quadratic left-Alia algebra $(\mathfrak{g}\ltimes_{\mathrm{ad}^{*},2\mathrm{ad}^{*}}\mathfrak{g}^{*},\mathcal{B}_{d})$, where $\mathrm{ad}:\mathfrak{g}\rightarrow\mathrm{End}(\mathfrak{g})$ is the adjoint representation of $(\mathfrak{g},[\cdot,\cdot])$.
\end{rmk}
We study the tensor forms of nondegenerate symmetric invariant bilinear forms on left-Alia algebras.

\begin{defi}
	Let $(A,[\cdot,\cdot])$ be a left-Alia algebra and
	$h:A\rightarrow \mathrm{End}(A\otimes A)$ be a linear map given by
	\begin{equation}
		h(x)=(\mathcal{R}_{[\cdot,\cdot]}-\mathcal{L}_{[\cdot,\cdot]})(x)\otimes\mathrm{id}-\mathrm{id}\otimes\mathcal{R}_{[\cdot,\cdot]}(x),
		\;\forall x\in A.
	\end{equation}
	
	An element $r\in A\otimes A$ is called \textbf{{invariant}} on $(A,[\cdot,\cdot])$ if $h(x)r=0$ for all $x\in A$.
\end{defi}

\delete{
	Let $V$ and $A$ be vector spaces. We identify a linear map $f:V\rightarrow A$ as an element $f_{\sharp}\in V^{*}\otimes A$ by
	\begin{equation}\label{eq:tensor form}
		\langle f_{\sharp}, u\otimes a^{*}\rangle=\langle f(u), a^{*}\rangle,\;\forall u\in V, a^{*}\in A^{*}.
	\end{equation}
	Now suppose $\mathcal{B}$ is a nondegenerate bilinear form on $A$ and $\mathcal{B}^{\natural}:A\rightarrow A^{*}$
	be the corresponding map given by \eqref{eq:bn}.
	We set $\widetilde{\mathcal{B}}=(\mathcal{B}^{\natural^{-1}})_{\sharp}\in A\otimes A$.
	Then we have the following result.}

\begin{pro}
	Let $(A,[\cdot,\cdot])$ be a left-Alia algebra.
	Suppose that $\mathcal{B}$ is a nondegenerate bilinear form on $A$ and $\mathcal{B}^{\natural}:A\rightarrow A^{*}$
	is the corresponding map given by \eqref{eq:bn}.
	Set $\widetilde{\mathcal{B}} \in A\otimes A$ using
	\begin{equation}
		\langle \widetilde{\mathcal{B}}, a^{*}\otimes b^{*}\rangle=\langle \mathcal{B}^{\natural^{-1}}(a^{*}), b^{*}\rangle,\;\forall a^{*},b^{*}\in A^{*}.
	\end{equation}
	
	Then, $(A,[\cdot,\cdot],\mathcal{B})$ is a quadratic left-Alia algebra if and only if $\widetilde{\mathcal{B}}$ is symmetric and invariant \linebreak  on $(A,[\cdot,\cdot])$.
\end{pro}
\begin{proof}
	It is clear that $\mathcal{B}$ is symmetric if and only if $ \widetilde{\mathcal{B}} $ is symmetric.
	Let $x,y,z\in A$ and $a^{*}=\mathcal{B}^{\natural}(x), c^{*}=\mathcal{B}^{\natural}(z)$.
	Under the symmetric assumption, we have
	{\small
		\begin{eqnarray*}
			&&\mathcal{B}([x,y],z)=\langle [x,y],\mathcal{B}^{\natural}(z)\rangle=\langle [\mathcal{B}^{\natural^{-1}}(a^{*}), y], c^{*}\rangle\\
			&&=-\langle \mathcal{B}^{\natural^{-1}}(a^{*}), \mathcal{R}^{*}_{[\cdot,\cdot]}(y)c^{*}\rangle=-\langle \widetilde{\mathcal{B}}, a^{*}\otimes\mathcal{R}^{*}_{[\cdot,\cdot]}(y)c^{*}\rangle=\langle \big(\mathrm{id}\otimes\mathcal{R}_{[\cdot,\cdot]}(y)\big)\widetilde{\mathcal{B}}, a^{*}\otimes c^{*}\rangle,\\
			&&\mathcal{B}(x,[z,y]-[y,z])=\langle \mathcal{B}^{\natural}(x), [z,y]-[y,z]\rangle=\langle a^{*},[\mathcal{B}^{\natural^{-1}}(c^{*}),y]-[y,\mathcal{B}^{\natural^{-1}}(c^{*})] \rangle\\
			&&=\langle(\mathcal{L}^{*}_{[\cdot,\cdot]}-\mathcal{R}^{*}_{[\cdot,\cdot]})(y)a^{*}, \mathcal{B}^{\natural^{-1}}(c^{*})\rangle
			=\langle \widetilde{\mathcal{B}}, c^{*}\otimes (\mathcal{L}^{*}_{[\cdot,\cdot]}-\mathcal{R}^{*}_{[\cdot,\cdot]})(y)a^{*}\rangle\\
			&&= \langle \big( (\mathcal{R}_{[\cdot,\cdot]}-\mathcal{L}_{[\cdot,\cdot]})(y)\otimes\mathrm{id} \big) \widetilde{\mathcal{B}}, a^{*}\otimes c^{*}\rangle,
	\end{eqnarray*}}
	that is, \eqref{eq:quad} holds if and only if $h(y)\widetilde{\mathcal{B}}=0$ for all $y\in A$.
	Hence, the conclusion follows.
\end{proof}

\section{Manin Triples of Left-Alia Algebras and Left-Alia Bialgebras}

In this section, we introduce the notions of Manin triples of left-Alia algebras and left-Alia bialgebras.
We show that they are equivalent structures via specific matched pairs of left-Alia algebras.

\subsection{Manin Triples of Left-Alia Algebras}

\vspace{6pt}
\begin{defi}\label{Manin}
	Let $(A,[\cdot,\cdot]_A )$ and $(A^{*},[\cdot,\cdot]_{A^{*}} )$ be left-Alia algebras.
	Assume  that there is a left-Alia algebra structure $(d=A \oplus A^{*},[\cdot,\cdot]_d )$ on $A\oplus A^{*}$  which contains $(A,[\cdot,\cdot]_A )$ and $(A^{*},[\cdot,\cdot]_{A^{*}} )$ as left-Alia subalgebras.
	Suppose that  the natural nondegenerate symmetric bilinear form $\mathcal{B}_{d}$, given by \eqref{eq:Bd}, is invariant on $(A \oplus A^{*},[\cdot,\cdot]_d )$, that is, $(A \oplus A^{*},[\cdot,\cdot]_{d},\mathcal{B}_{d} )$ is a quadratic left-Alia algebra.
	Then, we say that $\big( (A\oplus A^{*},[\cdot,\cdot]_{d},\mathcal{B}_{d}), A, A^{*} \big)$ is a \textbf{{Manin triple of left-Alia algebras}}.
\end{defi}

Recall~\cite{Bai2010} that a \textbf{{double construction of commutative Frobenius algebras}} \linebreak  $\big( (A\oplus A^{*},\cdot_{d}$,
$\mathcal{B}_{d}),A, A^{*}    \big)$ is a commutative associative algebra $(A\oplus A^{*},\cdot_{d})$ containing $(A,\cdot_{A})$ and $(A^{*},\cdot_{A^{*}})$ as commutative associative subalgebras, such that the natural nondegenerate symmetric bilinear form $\mathcal{B}_{d}$ given by \eqref{eq:Bd} is invariant on $(A\oplus A^{*},\cdot_{d})$.
Now, we show that double constructions of commutative Frobenius algebras with  linear maps naturally give rise to Manin triples of left-Alia algebras.

\begin{cor}
	Let $\big( (A\oplus A^{*},\cdot_{d},\mathcal{B}_{d}), A, A^{*} \big)$ be a double construction of commutative Frobenius algebras.
	Suppose that $P:A\rightarrow A$ and $Q^{*}:A^{*}\rightarrow A^{*}$ are linear maps.
	Then, there is a Manin triple of left-Alia algebras
	$\big( (A\oplus A^{*},[\cdot,\cdot]_{d},\mathcal{B}_{d}),  A ,  A^{*} \big)$ given by
	\begin{eqnarray*}
		&&[x+a^{*},y+b^{*}]_{d}=(x+a^{*})\cdot_{d}\big(P(y)+Q^{*}(b^{*})\big)-(Q+P^{*})\big( (x+a^{*})\cdot_{d}(y+b^{*}) \big),\\
		&&[x,y]_{A}=x\cdot_{A}P(y)-Q(x\cdot_{A}y),\; [a^{*},b^{*}]_{A^{*}}=a^{*}\cdot_{A^{*}}Q^{*}(b^{*})-P^{*}(a^{*}\cdot_{A^{*}}b^{*}),
	\end{eqnarray*}
	for all $x,y\in A, a^{*},b^{*}\in A^{*}$.
\end{cor}
\begin{proof}
	The adjoint map of $P+Q^{*}$ with respect to $\mathcal{B}_{d}$ is $Q+P^{*}$.
	Hence, the conclusion follows from Proposition~\ref{pro:2.8} by taking $f=P+Q^{*}$.
\end{proof}

\begin{thm}\label{thm:1}
	Let $(A,[\cdot,\cdot]_A )$ and $(A^{*},[\cdot,\cdot]_{A^{*}} )$ be left-Alia algebras.
	Then, there is a Manin triple of left-Alia algebras $\big( (A\oplus A^{*},[\cdot,\cdot]_{d},\mathcal{B}_{d}), A, A^{*} \big)$ if and only if
	$$\big ( (A,[\cdot,\cdot]_A ), (A^{*},[\cdot,\cdot]_{A^{*}} ), \mathcal{L}_{[\cdot,\cdot]_A}^*, \mathcal{L}_{[\cdot,\cdot]_A}^*-\mathcal{R}^{*}_{[\cdot,\cdot]_A},
	\mathcal{L}^{*}_{[\cdot,\cdot]_{A^*}}, \mathcal{L}^{*}_{[\cdot,\cdot]_{A^*}}-\mathcal{R}^{*}_{[\cdot,\cdot]_{A^*}}
	\big)$$
	is a matched pair of left-Alia algebras.
\end{thm}
\begin{proof}
	Let $\big( (A\oplus A^{*},[\cdot,\cdot]_d,\mathcal{B}_{d}), A, A^{*} \big)$ be a Manin triple of left-Alia algebras.
	For all $x,y\in A, a^{*},b^{*}\in A^{*}$, we have
	\begin{eqnarray*}
		\mathcal{B}_{d}([x, b^{*}]_d,y)&\overset{\eqref{eq:quad}}{=}&-\mathcal{B}(b^{*}, [x,y]_A)
		=-\langle b^{*}, [x,y]_A\rangle
		=\langle \mathcal{L}^{*}_{[\cdot,\cdot]_A}(x)b^{*}, y\rangle
		=\mathcal{B}_{d}\big(\mathcal{L}^{*}_{[\cdot,\cdot]_A}(x)b^{*},y\big),\\
		\mathcal{B}_{d}([x, b^{*}]_d,a^{*})&\overset{\eqref{eq:quad}}{=}&\mathcal{B}_{d}(x, [a^{*}, b^{*}]_{A^*}-[b^{*},a^{*}]_{A^*})
		=\langle x, [a^{*},b^{*}]_{A^*}-[b^{*},a^{*}]_{A^*}\rangle \\
		&=&\langle (\mathcal{L}^{*}_{[\cdot,\cdot]_{A^*}}-\mathcal{R}^{*}_{[\cdot,\cdot]_{A^*}})(b^{*})x, a^{*}\rangle=\mathcal{B}_{d}\big( (\mathcal{L}^{*}_{[\cdot,\cdot]_{A^*}}-\mathcal{R}^{*}_{[\cdot,\cdot]_{A^*}})(b^{*})x, a^{*}  \big).
	\end{eqnarray*}

	Thus,
	\begin{equation*}
		\mathcal{B}_{d}(  [x,b^{*}]_d, y+a^{*} )=\mathcal{B}_{d}\big(  (\mathcal{L}^{*}_{[\cdot,\cdot]_{A^*}}-\mathcal{R}^{*}_{[\cdot,\cdot]_{A^*}})(b^{*})x+\mathcal{L}^{*}_{[\cdot,\cdot]_A}(x)b^{*} , y+a^{*} \big)
	\end{equation*}
	and, by the nondegeneracy of $\mathcal{B}_{d}$, we have
	\begin{equation*}
		[x, b^{*}]_d
		=(\mathcal{L}^{*}_{[\cdot,\cdot]_{A^*}}-\mathcal{R}^{*}_{[\cdot,\cdot]_{A^*}})(b^{*})x+\mathcal{L}^{*}_{[\cdot,\cdot]_A}(x)b^{*}.
	\end{equation*}
	
	Similarly,
	\begin{equation*}
		[y,a^{*}]_d=(\mathcal{L}^{*}_{[\cdot,\cdot]_{A }}-\mathcal{R}^{*}_{[\cdot,\cdot]_{A }})(y)a^{*}+\mathcal{L}^{*}_{[\cdot,\cdot]_{A^*}}(a^{*})y.
	\end{equation*}
	
	Therefore, we have
	\begin{align}
		[x+a^{*},y+b^{*}]_d&=[x, y]_A+\mathcal{L}^{*}_{[\cdot,\cdot]_{A^*}}(a^{*})y+(\mathcal{L}^{*}_{[\cdot,\cdot]_{A^*}}-\mathcal{R}^{*}_{[\cdot,\cdot]_{A^*}})(b^{*})x\nonumber\\
		&+[a^{*}, b^{*}]_{A^*}+\mathcal{L}^{*}_{[\cdot,\cdot]_A}(x)b^{*}+(\mathcal{L}^{*}_{[\cdot,\cdot]_{A }}-\mathcal{R}^{*}_{[\cdot,\cdot]_{A }})(y)a^{*}.\label{eq:mp dual rep}
	\end{align}
	
	Hence, $\big ( (A,[\cdot,\cdot]_A ), (A^{*},[\cdot,\cdot]_{A^*} ), \mathcal{L}^{*}_{[\cdot,\cdot]_A}, \mathcal{L}^{*}_{[\cdot,\cdot]_A}-\mathcal{R}^{*}_{[\cdot,\cdot]_A},
	\mathcal{L}^{*}_{[\cdot,\cdot]_{A^*}}, \mathcal{L}^{*}_{[\cdot,\cdot]_{A^*}}-\mathcal{R}^{*}_{[\cdot,\cdot]_{A^*}}
	\big)$ is a \linebreak matched pair of left-Alia algebras.
	
	Conversely, if $\big ( (A,[\cdot,\cdot]_A ), (A^{*},[\cdot,\cdot]_{A^*} ), \mathcal{L}^{*}_{[\cdot,\cdot]_A}, \mathcal{L}^{*}_{[\cdot,\cdot]_A}-\mathcal{R}^{*}_{[\cdot,\cdot]_A},
	\mathcal{L}^{*}_{[\cdot,\cdot]_{A^*}}, \mathcal{L}^{*}_{[\cdot,\cdot]_{A^*}}-\mathcal{R}^{*}_{[\cdot,\cdot]_{A^*}}
	\big)$ is a matched pair of left-Alia algebras, then it is straightforward to check that  $\mathcal{B}_{d}$ is invariant on the left-Alia algebra $(A\oplus A^{*},[\cdot,\cdot]_{d})$ given by \eqref{eq:mp dual rep}.
\end{proof}

\subsection{Left-Alia Bialgebras}

\vspace{6pt}
\begin{defi}\label{defi:anti-pre-Lie coalgebras}
	A  \textbf{{left-Alia coalgebra}} is a pair, $(A,\delta)$, such that $A$ is a vector space and \linebreak  $\delta:A\rightarrow A\otimes A$ is a co-multiplication satisfying
	\begin{equation}\label{eq:defi:coalgebra}
		(\mathrm{id}^{\otimes 3}+\xi+\xi^{2})(\tau\otimes \mathrm{id}-\mathrm{id}^{\otimes 3})(\delta\otimes \mathrm{id})\delta=0,
	\end{equation}
	where $\tau(x\otimes y)=y\otimes x$ and $\xi(x\otimes y\otimes z)=y\otimes z\otimes x$ for all $x,y,z\in A$.
\end{defi}
\begin{pro}\label{pro:coalgebras}
	Let $A$ be a vector space and $\delta:A\rightarrow A\otimes A$ be a co-multiplication.
	Let $[\cdot,\cdot]_{A^*} :A^{*}\otimes A^{*}\rightarrow A^{*}$ be the linear dual of $\delta$, that is,
	\begin{equation}
		\langle [a^{*}, b^{*}]_{A^*},x\rangle=\langle\delta^{*}(a^{*}\otimes b^{*}),x\rangle=\langle a^{*}\otimes b^{*},\delta(x)\rangle, \;\;\forall a^{*},b^{*}\in A^{*}, x\in A.
	\end{equation}
	
	Then, $(A,\delta)$ is a left-Alia coalgebra if and only if $(A^{*},[\cdot,\cdot]_{A^{*}} )$ is a left-Alia algebra.
\end{pro}
\begin{proof}
	For all $x\in A, a^{*},b^{*},c^{*}\in A^{*}$,   we have
\begin{eqnarray*}
		\langle [[a^{*}, b^{*}]_{A^*},c^{*}]_{A^*}-[[b^{*},a^{*}]_{A^*},c^{*}]_{A^*},x\rangle&=&
		\langle \delta^{*}(\delta^{*}\otimes\mathrm{id})(\mathrm{id}^{\otimes 3}-\tau\otimes\mathrm{id})a^{*}\otimes b^{*}\otimes c^{*},x\rangle\\
		&=&\langle a^{*}\otimes b^{*}\otimes c^{*},(\mathrm{id}^{\otimes 3}-\tau\otimes\mathrm{id})(\delta\otimes\mathrm{id})\delta(x)\rangle,\\
		\langle [[b^{*},c^{*}]_{A^*},a^{*}]_{A^*}-[[c^{*}, b^{*}]_{A^*}, a^{*}]_{A^*},x\rangle&=&
		\langle b^{*}\otimes c^{*}\otimes a^{*},(\mathrm{id}^{\otimes 3}-\tau\otimes\mathrm{id})(\delta\otimes\mathrm{id})\delta(x)\rangle\\
		&=&\langle a^{*}\otimes b^{*}\otimes c^{*},\xi^{2}(\mathrm{id}^{\otimes 3}-\tau\otimes\mathrm{id})(\delta\otimes\mathrm{id})\delta(x)\rangle,\\
		\langle [[c^{*},a^{*}]_{A^*}, b^{*}]_{A^*}-[[a^{*},c^{*}]_{A^*},b^{*}]_{A^*},x\rangle&=&
		\langle c^{*}\otimes a^{*}\otimes b^{*},(\mathrm{id}^{\otimes 3}-\tau\otimes\mathrm{id})(\delta\otimes\mathrm{id})\delta(x)\rangle\\
		&=&\langle a^{*}\otimes b^{*}\otimes c^{*},\xi (\mathrm{id}^{\otimes 3}-\tau\otimes\mathrm{id})(\delta\otimes\mathrm{id})\delta(x)\rangle.
	\end{eqnarray*}
	
	Hence, \eqref{eq:0-Alia} holds for $(A^{*},[\cdot,\cdot]_{A^*})$
	if and only if \eqref{eq:defi:coalgebra} holds.
\end{proof}

\begin{defi}\label{bialgebra}
	A \textbf{{left-Alia bialgebra}} is a triple $(A,[\cdot,\cdot],\delta)$, such that
	$(A,[\cdot,\cdot])$ is  a left-Alia algebra, $(A,\delta)$ is a left-Alia coalgebra and the following equation holds:

%\vspace{-6pt}\begin{adjustwidth}{-\extralength}{0cm}
%\centering %% If there is a figure in wide page, please release command \centering
	\begin{equation}\label{eq:bialg}
		(\tau-\mathrm{id}^{2})\big( \delta([x,y]-[y,x])+(\mathcal{R}_{[\cdot,\cdot]}(x)\otimes\mathrm{id})\delta(y)-(\mathcal{R}_{[\cdot,\cdot]}(y)\otimes\mathrm{id})\delta(x) \big)=0,\;\forall x,y\in A.
	\end{equation}
%\end{adjustwidth}
\end{defi}

\begin{thm}\label{thm:2}
	Let $(A,[\cdot,\cdot]_A )$ be a left-Alia algebra. Suppose that there is a left-Alia algebra structure $(A^{*},[\cdot,\cdot]_{A^{*}} )$ on the dual space $A^{*}$, and $\delta:A\rightarrow A\otimes A$ is the linear dual of $[\cdot,\cdot]_{A^*} $.
	Then, $\big ( (A,[\cdot,\cdot]_A ), (A^{*},[\cdot,\cdot]_{A^*} ), \mathcal{L}^{*}_{[\cdot,\cdot]_A}, \mathcal{L}^{*}_{[\cdot,\cdot]_A}-\mathcal{R}^{*}_{[\cdot,\cdot]_A},
	\mathcal{L}^{*}_{[\cdot,\cdot]_{A^*}}, \mathcal{L}^{*}_{[\cdot,\cdot]_{A^*}}-\mathcal{R}^{*}_{[\cdot,\cdot]_{A^*}}
	\big)$ is a matched pair of left-Alia algebras if and only if $(A,[\cdot,\cdot]_A,\delta)$ is a left-Alia bialgebra.
\end{thm}

\newpage
\begin{proof}
	For all $x,y\in A, a^{*},b^{*}\in A^{*}$, we have
%\vspace{-12pt}\begin{adjustwidth}{-\extralength}{0cm}
\begin{eqnarray*}
		\langle(\mathcal{L}^{*}_{[\cdot,\cdot]_{A^*}}-\mathcal{R}^{*}_{[\cdot,\cdot]_{A^*}})(a^{*})([x,y]_{A}-[y,x]_{A}), b^{*}\rangle&=&\langle [x,y]_{A}-[y,x]_{A}, [b^{*},a^{*}]_{A^*}-[a^{*},b^{*}]_{A^*}\rangle\\
		&=&\langle (\tau-\mathrm{id}^{\otimes 2})\delta([x,y]_A-[y,x]_A), a^{*}\otimes b^{*}\rangle,\\
		\langle[\mathcal{R}^{*}_{[\cdot,\cdot]_{A^*}}(a^{*})y,  x]_A, b^{*}\rangle&=&-\langle \mathcal{R}^{*}_{[\cdot,\cdot]_{A^*}}(a^{*})y, \mathcal{R}^{*}_{[\cdot,\cdot]_A  }(x)b^{*}\rangle\\
		&=&\langle y, [\mathcal{R}^{*}_{ [\cdot,\cdot]_A  }(x)b^{*}, a^{*}]_{A^*}\rangle\\
		&=&-\langle \big( \mathcal{R}_{[\cdot,\cdot]_A }(x)\otimes\mathrm{id}\big)\delta(y) , b^{*}\otimes a^{*}\rangle\\
		&=&-\langle \tau\big( \mathcal{R}_{[\cdot,\cdot]_A  }(x)\otimes\mathrm{id} \big)\delta(y),a^{*}\otimes b^{*}\rangle,\\
		-\langle[\mathcal{R}^{*}_{[\cdot,\cdot]_{A^*}}(a^{*})y,  x]_A, b^{*}\rangle&=&\langle \tau\big( \mathcal{R}_{[\cdot,\cdot]_A }(y)\otimes\mathrm{id} \big)\delta(x),a^{*}\otimes b^{*}\rangle,\\
		-\langle\mathcal{L}^{*}_{[\cdot,\cdot]_{A^*}}\big( \mathcal{R}^{*}_{[\cdot,\cdot]_A}(y)a^{*}  \big)x,b^{*}\rangle&=&\langle x, [\mathcal{R}^{*}_{[\cdot,\cdot]_A}(y)a^{*}, b^{*}]_{A^*}\rangle\\
		&=&-\langle \big( \mathcal{R}_{[\cdot,\cdot]_A  }(y)\otimes\mathrm{id} \big)\delta(x), a^{*}\otimes b^{*}\rangle,\\
		\langle\mathcal{L}^{*}_{[\cdot,\cdot]_{A^*}}\big( \mathcal{R}^{*}_{[\cdot,\cdot]_A}(x)a^{*})y \big),b^{*}\rangle
		&=& \langle \big( \mathcal{R}_{[\cdot,\cdot]_A }(x)\otimes\mathrm{id} \big)\delta(y), a^{*}\otimes b^{*}\rangle.
	\end{eqnarray*}
%\end{adjustwidth}

	Thus, \eqref{eq:bialg} holds if and only if \eqref{eq:defi:matched pairs1} holds for $l_{A}=\mathcal{L}^{*}_{[\cdot,\cdot]_A},\; r_{A}=\mathcal{L}^{*}_{[\cdot,\cdot]_A}-\mathcal{R}^{*}_{[\cdot,\cdot]_A}, \linebreak  \;l_{B}=\mathcal{L}^{*}_{[\cdot,\cdot]_{A^*}},\; r_{B}=\mathcal{L}^{*}_{[\cdot,\cdot]_{A^*}}-\mathcal{R}^{*}_{[\cdot,\cdot]_{A^*}}$.
	Similarly, \eqref{eq:bialg} holds if and only if \eqref{eq:defi:matched pairs2} holds for $l_{A}=\mathcal{L}^{*}_{[\cdot,\cdot]_A},\; r_{A}=\mathcal{L}^{*}_{[\cdot,\cdot]_A}-\mathcal{R}^{*}_{[\cdot,\cdot]_A},\;l_{B}=\mathcal{L}^{*}_{[\cdot,\cdot]_{A^*}},\; r_{B}=\mathcal{L}^{*}_{[\cdot,\cdot]_{A^*}}-\mathcal{R}^{*}_{[\cdot,\cdot]_{A^*}}$.
	Hence, the conclusion follows.
\end{proof}

Summarizing Theorems~\ref{thm:1} and~\ref{thm:2} , we have the following corollary:

\begin{cor}\label{cor:4.8}
	Let $(A,[\cdot,\cdot]_A )$ be a left-Alia algebra. Suppose that there is a left-Alia algebra structure $(A^{*},[\cdot,\cdot]_{A^{*}} )$ on the dual space $A^{*}$, and $\delta:A\rightarrow A\otimes A$ is the linear dual of $[\cdot,\cdot]_{A^*}$. Then, the following conditions are equivalent:
	\begin{enumerate}
		\item[(a)] There is a Manin triple of left-Alia algebras $\big( (d=A\oplus A^{*},[\cdot,\cdot]_{d},\mathcal{B}_{d}), A, A^{*} \big)$.
		\item[(b)] $\big ( (A,[\cdot,\cdot]_{A} ), (A^{*},[\cdot,\cdot]_{A^*} ), \mathcal{L}^{*}_{[\cdot,\cdot]_A}, \mathcal{L}^{*}_{[\cdot,\cdot]_A}-\mathcal{R}^{*}_{[\cdot,\cdot]_A},
		\mathcal{L}^{*}_{[\cdot,\cdot]_{A^*}}, \mathcal{L}^{*}_{[\cdot,\cdot]_{A^*}}-\mathcal{R}^{*}_{[\cdot,\cdot]_{A^*}}
		\big)$ is a matched pair of left-Alia algebras.
		\item[(c)] $(A,[\cdot,\cdot]_A,\delta)$ is a left-Alia bialgebra.
	\end{enumerate}
\end{cor}

\begin{ex}
	Let $(A,[\cdot,\cdot]_{A})$ be the three-dimensional left-Alia algebra given in Example~\ref{ex:3-left}.
	
	Then, there is a left-Alia bialgebra $(A,[\cdot,\cdot]_{A},\delta)$ with a non-zero co-multiplication $\delta$ on $A$, given by
	\begin{equation}\label{eq:co}
		\delta(e_{1})=e_{1}\otimes e_{1}.
	\end{equation}
	
	Then, by Corollary~\ref{cor:4.8}, there is a Manin triple $\big( (A\oplus A^{*},[\cdot,\cdot],\mathcal{B}_{d}), A, A^{*}\big)$. %(A,[\cdot,\cdot]),(A^{*},[\cdot,\cdot]_{A^{*}})\big)$.
	Here, the multiplication $[\cdot,\cdot]_{A^{*}}$ on $A^{*}$ is given through $\delta$ by \eqref{eq:co}, that is,
	\begin{equation*}
		[e^{*}_{1},e^{*}_{1}]_{A^{*}}=e^{*}_{1},
	\end{equation*}
	and the multiplication $[\cdot,\cdot]$ on $A\oplus A^{*}$ is given by \eqref{eq:mp dual rep}.
	Moreover, $\big ( (A,[\cdot,\cdot]_{A} ), (A^{*},[\cdot,\cdot]_{A^*} ), \mathcal{L}^{*}_{[\cdot,\cdot]_A}$,
	$ \mathcal{L}^{*}_{[\cdot,\cdot]_A}-\mathcal{R}^{*}_{[\cdot,\cdot]_A},
	\mathcal{L}^{*}_{[\cdot,\cdot]_{A^*}}, \mathcal{L}^{*}_{[\cdot,\cdot]_{A^*}}-\mathcal{R}^{*}_{[\cdot,\cdot]_{A^*}}
	\big)$ is a matched pair of left-Alia algebras.
\end{ex}

\noindent{\bf Acknowledgements.} The fourth author would like to thank Prof.Lam Siu-Por, Dr.Li Yu and Dr.Wang Chuijia for helpful discussion on invariant theory. The fourth author acknowledges support from the NSF China (12101328) and NSF China (12371039).

\end{document}